\documentclass[a4paper]{article}

\usepackage[latin1]{inputenc}
\usepackage{amsmath,amsfonts,amssymb}
\usepackage[autolanguage]{numprint}
\usepackage[usenames]{color}
\usepackage{graphicx}
\usepackage{subfig}
\usepackage[english]{babel}

\graphicspath{{figures/}}

\newcommand{\R}{\mathbb{R}}
\newcommand{\Te}{\mathcal{T}}
\renewcommand{\S}[1]{\mathcal{S}_{#1}}
\newcommand{\M}[1]{\mathcal{M}_{#1}}
\newcommand{\F}[1]{\mathcal{F}_{#1}}
\newcommand{\mean}[1]{\operatorname{\mathbb{E}}#1}
\newcommand{\prob}[1]{\operatorname{\mathbb{P}}\left(#1\right)}

\newcommand{\sve}[1]{V(#1)}
\newcommand{\sveint}[1]{\overset{\circ}{V}(#1)}
\newcommand{\nve}[1]{n_\text{v}(#1)}
\newcommand{\nveint}[1]{\overset{\circ}{n}_\text{v}(#1)}
\newcommand{\sed}[1]{E(#1)}

\newcommand{\ned}[1]{n_\text{e}(#1)}

\newcommand{\sce}[1]{C(#1)}

\newcommand{\nce}[1]{n_\text{c}(#1)}

\newcommand{\sse}[1]{S(#1)}

\newcommand{\nse}[1]{n_\text{s}(#1)}
\newcommand{\nseint}[1]{\overset{\circ}{n}_\text{s}(#1)}
\newcommand{\nbseint}[1]{\overset{\circ}{n}_{\text{s,b}}\left(#1\right)}
\newcommand{\nnbseint}[1]{\overset{\circ}{n}_{\text{s,nb}}\left(#1\right)}
\newcommand{\nttl}[1]{n_{\text{t}}(#1)}
\newcommand{\st}{\text{s}} 
\newcommand{\mt}{\text{m}} 
\newcommand{\ft}{\text{f}} 

\newcommand{\PP}{P}
\newcommand{\TT}{\mathbf{T}}  
\newcommand{\cT}{\mathcal{T}} 

\newcommand{\CSplitSpace}{\mathcal{C}_{\text{s}}}
\newcommand{\CSplitMeas}{{C}_{\text{s}}}
\newcommand{\PSplit}{{P}_{\text{s}}}
\newcommand{\CFlipSpace}{\mathcal{C}_{\text{f}}}
\newcommand{\CFlipMeas}{{C}_{\text{f}}}
\newcommand{\PFlip}{{P}_{\text{f}}}

\newtheorem{theorem}{Theorem}
\newtheorem{proposition}{Proposition}
\newtheorem{definition}{Definition}
\newtheorem{corollary}{Corollary}

\newenvironment{proof}{\noindent{\bf Proof} }{\hfill $\square$ \\}
\newcounter{exampleno}
\newenvironment{example}{\refstepcounter{exampleno}\noindent{\bf Example \arabic{exampleno}} }{\hfill $\square$ \\}

\title{A completely random T-tessellation model and Gibbsian extensions}
\author{Kiên Kiêu\textsuperscript{\dag}, Katarzyna Adamczyk-Chauvat\textsuperscript{\dag}, Hervé Monod\textsuperscript{\dag},\\ Radu S.\ Stoica\textsuperscript{\ddag,\P}\vspace{0.5cm}\\
\textsuperscript{\dag}UR 341 Mathématiques et Informatique Appliquées, INRA\\ F78350 Jouy-en-Josas\\
\textsuperscript{\ddag}UMR 8524 Laboratoire Paul Painlevé, Université Lille 1\\ F59491 Villeneuve-d'Ascq\\
\textsuperscript{\P}Institut de mécanique céleste et de calcul des éphémérides,\\ Observatoire de Paris\\
F75014 Paris, France}

\begin{document}
\maketitle 
\begin{abstract}
  In their 1993 paper, Arak, Clifford and Surgailis discussed a new model of random planar graph. As a particular case, that model yields tessellations with only T-vertices (T-tessellations). Using a similar approach involving Poisson lines, a new model of random T-tessellations is proposed. Campbell measures, Papangelou kernels and Georgii-Nguyen-Zessin formulae are translated from point process theory to random T-tessellations. It is shown that the new model shows properties similar to the Poisson point process and can therefore be considered as a completely random T-tessellation. Gibbs variants are introduced leading to models of random T-tessellations where selected features are controlled. Gibbs random T-tessellations are expected to better represent observed tessellations. As numerical experiments are a key tool for investigating Gibbs models, we derive a simulation algorithm of the Metropolis-Hastings-Green family.
\end{abstract}

\tableofcontents

\section{Introduction}
\label{sec:intro}

Random tessellations are attractive mathematical objects, from both theoretical and practical points of view. The study of the mathematical properties of these objects is still leading to open problems, while the range of applications covers a broad panel of scientific domains such as astronomy, geophysics, image processing or environmental sciences~\cite{Lant02,MollStoy07,leber}.

When modelling real-world structures, one aims at flexible classes of random models able to represent a wide range of spatial patterns. Gibbsian point processes combined with Voronoï diagrams offer such an attractive approach. The class of Gibbs point processes is enriched continuously (hard-core, Strauss, area-interaction, Quermass-interaction). Random tessellations are obtained as Voronoï diagrams associated with the germs yielded by such Gibbs processes, see e.g. \cite{dereudre-2011}. Hence a large class of models for random tessellations is made available for applications. Our aim is to sketch a similar theoretical framework for other types of tessellations which are not germ-based. We will focus on \emph{T-tessellations} (tessellations with only T-vertices) which are naturally built using geometrical operators (splits, merges and flips) instead of a germ-based procedure.

Our approach is closely related to the one used by Arak, Clifford and Surgailis in their paper~\cite{ArakClifSurg93} on a random planar graph model. In particular, a key ingredient is the Poisson line process. As a first step, this paper introduces a new model of random T-tessellations which can be considered as a completely random model. Then Gibbs variations are proposed and a general algorithm for simulating them is described. 

Throughout the paper, we focus on the case where the domain of interest is bounded. Extension of Gibbs models for T-tessellations of the whole plane remains an open problem at this stage. 

Section \ref{sec:space} provides definitions, notations and basic results about T-tes\-sel\-la\-tions. The completely random T-tessellation model is discussed in Section~\ref{sec:random}. This section also introduces for arbitrary random T-tessellations Campbell measures, Papangelou kernels which are widely used in point process theory. Our new model can be considered as a T-tessellation analog to the Poisson point process. This claim is based on Georgii-Nguyen-Zessin type formulae. Therefore the T-tessellation model introduced in Section~\ref{sec:random} is referred to as a completely random T-tessellation. Gibbsian extensions are discussed in Section \ref{sec:gibbs:extension} together with some examples. One example is also a particular case of Arak-Clifford-Surgailis random graph model when its parameters are chosen in order to yield a T-tessellation. Georgii-Nguyen-Zessin formulae are provided for hereditary Gibbs models. In Section \ref{sec:simulation}, a simulation algorithm is proposed. The design of the algorithm follows the general principles of Metropolis-Hastings-Green algorithms, already widely used for Gibbs point processes. It involves three types of local operators: split, merge and flip. Conditions ensuring the convergence of the Markov chain to the target distribution are provided.

Below, as a notational convention, bold letters are used for denoting random variables. Detailed proofs of the main results are postponed in appendices.

\section{The space of T-tessellations}
\label{sec:space}

In this paper, we shall consider only bounded tessellations. The domain of interest is denoted by $D$. It is assumed to be compact. For sake of simplicity, it is supposed to be also convex and polygonal. Let $a(D)$, $l(D)$, $\ned{D}$ and $\nve{D}$ be respectively the area, the perimeter length, the numbers of sides and vertices of $D$.

A polygonal \emph{tessellation} of $D$ is a finite subdivision of $D$ into
polygonal sets (called \emph{cells}) with disjoint interiors. The
tessellation vertices are cell \emph{vertices}. \emph{Edges} are defined as line
segments contained in cell edges, ending at vertices and containing no other vertex between their ends. The whole set of tessellation edges can be
partitioned into maximal subsets of aligned and contiguous edges,
called \emph{segments}.

We shall focus on tessellations with only T-vertices. Such
tessellations are called here \emph{T-tessellations}. An example of
T-tessellation is shown in Figure~\ref{fig:T:tessellation}.

\begin{definition}
  A tessellation vertex is said to be a \emph{T-vertex} if it lies at the intersection
  of exactly three edges and if, among those three, two are aligned. 

  A polygonal tessellation of $D$ is called a \emph{T-tessellation} of $D$ if
  all its vertices, except those of $D$, are T-vertices and if it has no pair of
  distinct and aligned segments.
\end{definition}

\begin{figure}
  \centering
  \includegraphics{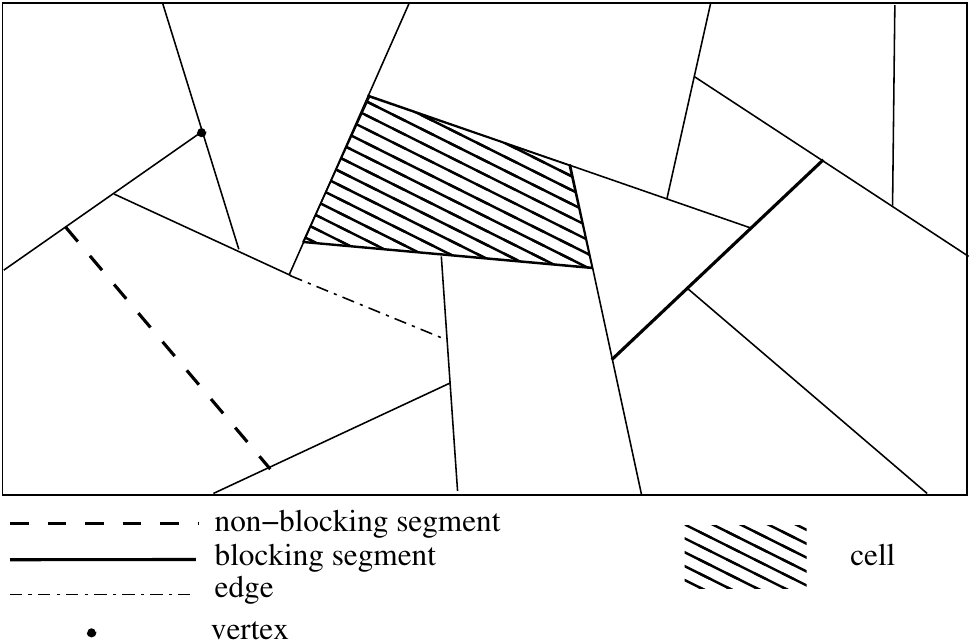}
  \caption{Example of a T-tessellation with its components.}
  \label{fig:T:tessellation}
\end{figure}

The space of T-tessellations of the domain $D$ is denoted by
$\cT$. For any T-tessellation $T\in\Te$, let $\sce{T}$ be the set of its cells, $\sve{T}$ the set of its vertices, $\sed{T}$ the set of its edges and $\sse{T}$ the set of its segments. The numbers of cells, vertices, edges and segments are denoted by $\nce{T}$, $\nve{T}$, $\ned{T}$ and $\nse{T}$. Vertices, edges and segments are said to be \emph{internal} when they do not lie on the boundary of the domain $D$. Sets and numbers of internal features are specified by adding a circle on top of symbols: e.g.\ $\sveint{T}$ and  $\nveint{T}$ for the set and the number of internal vertices. The numbers of vertices, edges and cells are closely related to the numbers of segments. The results given below hold under the condition that no internal segment ends at a vertex of the domain $D$. Any vertex is either a vertex of $D$ or the end of an internal segment. Since an internal segment has 2 ends,
\begin{equation}
  \label{eq:vertices:segments}
  \nve{T} = \nve{D} + 2\nseint{T}.
\end{equation}
The sum of vertex degrees in a graph is twice the number of its edges. Vertices in a T-tessellation have degree 2 (for a vertex of $D$) or 3. This observation combined with equality \eqref{eq:vertices:segments} yields the following identity
\begin{equation}
  \label{eq:edges:segments}
  \ned{T} = \nve{D} + 3\nseint{T}.
\end{equation}
Finally, using the Euler formula, one gets
\begin{equation}
  \label{eq:cells:segments}
  \nce{T} = \nseint{T} + 1.
\end{equation}

A T-tessellation yields a unique line pattern $L(T)$ defined as the set of lines supporting its segments. Conversely, for any finite pattern $L$ of lines hitting $D$, there are generally several T-tessellations $T$ such that $L(T)=L$. This set of tessellations is denoted $\cT(L)$ and its number of elements $\nttl{L}$. Kahn \cite[Theorem 3.4]{jkahn} proved the following inequality
\begin{equation}
  \label{eq:tessellation:number}
  \nttl{L} \leq C^{k} \left(\frac{k}{(\log k)^{1-\epsilon}}\right)^{k-k/\log k},
\end{equation}
where $k$ is the number of lines in $L$, $\epsilon$ is an arbitrary positive real number and $C$ a constant only depending on $\epsilon$. In the next sections, inequality \eqref{eq:tessellation:number} will be used for deriving further results involving so-called stable functionals. The stability condition defined below is similar to the one used for point processes (see \cite{geyer:_likel}). It ensures that unnormalized densities are integrable and define distributions on $\cT$ (see sections \ref{sec:random} and \ref{sec:gibbs:extension}).
\begin{definition}
  A non-negative functional $F$ on $\cT$ is said to be stable if there exists a real constant $K$ such that $F(T)\leq K^{\nseint{T}}$ for any $T\in\cT$.
\end{definition}

Below it is noticed that the whole space $\cT$ can be visited by means of three local and simple operators: splits, merges and flips.

A split is the division of a T-tessellation cell by a line segment,
see figures~\ref{fig:smf:a} and \ref{fig:smf:b}.  The continuous space of
all possible splits for a given T-tessellation $T$ is denoted by
$\S{T}$.

A merge is the deletion of a segment consisting of a single edge,
see figures~\ref{fig:smf:a} and \ref{fig:smf:c}. Such
a segment is said to be non-blocking. Other segments are called
blocking segments. The finite space of all possible merges for a given
T-tessellation $T$ is denoted by $\M{T}$. Note that $\M{T}$ may be
empty. For any T-tessellation $T$, there are $\nnbseint{T}$
possible merges, where $\nnbseint{T}$ is the number of internal
non-blocking segments of $T$.

A flip is the deletion of an edge at the end of a blocking segment
combined with the addition of a new edge. At one of the vertices of
the deleted edge, a segment is blocked. The new edge extends the
blocked segment until the next segment, see Figure~\ref{fig:smf:d}. There are $2\nbseint{T}$ possible flips,
where $\nbseint{T}$ is the number of internal blocking segments in
tessellation $T$. The finite space of all possible flips for a given
T-tessellation $T$ is denoted by $\F{T}$.
\begin{figure}
  \centering
  \subfloat[A T-tessellation.\label{fig:smf:a}]{\includegraphics{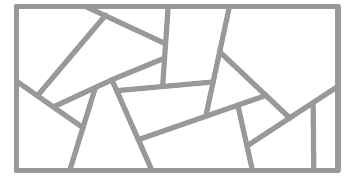}}\\
  \subfloat[Split: the new (black) edge is blocked by two existing segments.\label{fig:smf:b}]{\includegraphics{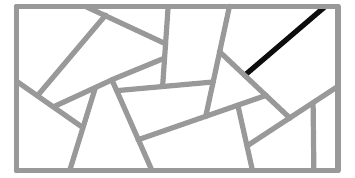}}
  \hspace{3mm}
  \subfloat[Merge: removal of a (light grey) non-blocking segment (in the top-left corner).\label{fig:smf:c}]{\includegraphics{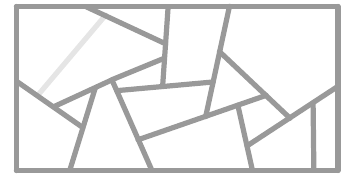}}
  \hspace{3mm}
  \subfloat[Flip: two segments are modified, one is shortened (suppressed edge in light grey) while the other one is extended (added edge in black).\label{fig:smf:d}]{\includegraphics{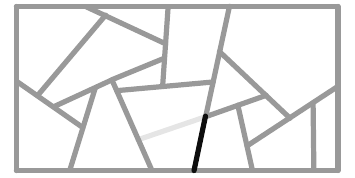}}
  \caption{Splits, merges and flips are local transformations of
    T-tessellations. Bottom row: transformations of the T-tessellation shown on the top row.}
  \label{fig:smf}
\end{figure}

Note that any split can be reversed by a merge (and vice-versa). Also
any flip can be reversed by another flip. Using these three types of
updates it is possible to visit any part of the space $\Te$. Below the
so-called empty tessellation refers to the trivial T-tessellation
$T_D$ that consists of a single cell extending to the whole domain
$D$. 
\begin{proposition}
\label{prop:smf:complet:1}
  Let $T$ be a non-empty T-tessellation. There exists a finite
  sequence of merges and flips which transforms $T$ into the empty
  tessellation $T_D$. Conversely there exists a finite sequence of splits
  and flips that transforms the empty tessellation $T_D$ into $T$.
\end{proposition}
In order to empty a given T-tessellation, one can start by removing every non-blocking segments (merges). If at some point there is no non-blocking segment left, one can make a blocking segment non-blocking by a finite series of flips. The reverse series of inverted operations (splits and flips) defines a path from the empty tessellation to the considered T-tessellation.
\begin{corollary}
  Consider any pair of distinct T-tessellations. There exists a finite
  sequence of splits, merges and flips which transforms one
  T-tessellation into the other one.
\end{corollary}

Next sections will involve uniform (probability) measures on the three spaces $\S{T}$, $\M{T}$ and $\F{T}$. Let us introduce these measures here. Since the spaces $\M{T}$
and $\F{T}$ are finite, any probability measure on those spaces is
just a finite set of probability values. In particular, one may
consider uniform distributions on the finite spaces $\M{T}$ and
$\F{T}$. Defining a probability measure on the continuous space
$\S{T}$ is less obvious.

Any split of $T$ can be represented by a pair $(c,l)$ where $c$ is a
cell of $T$ and $l$ is a line splitting $c$. Therefore $\S{T}$ can be
identified to the set of pairs $(c,l)$ such that $l$ splits $c$.
Consider the counting measure on the set of cells of $T$. The space of
planar lines is endowed with the unique (up to a
multiplicative constant) Haar measure (i.e.\ invariant under translations and rotations).
Here we consider the Haar measure scaled such that the mass of
the subset of lines hitting say the unit square is equal to $4/\pi$.
The space $\S{T}$ of splits applicable to $T$ is endowed with the
product of both measures restricted to the set of pairs $(c,l)$ such
that $l$ splits $c$. The infinitesimal volume element for that measure
at a split $S\in\S{T}$ is denoted by $dS$.  The measure $dS$ can be written as
\begin{equation}
  \label{eq:def:dS}
  dS = \sum_{c\in C(T)} \mathbf{1}_{\left\{c\cap l\neq\emptyset\right\}}\;dl,
\end{equation}
where $\mathbf{1}$ is the indicator function and $dl$ denotes the infinitesimal volume element of the Haar measure on the space of lines.
Using standard results
from integral geometry (Crofton formula, see e.g.\ \cite{stoyan-1995}), it can be checked that the total mass of that
measure is equal to
\begin{equation*}
  \frac{2l(T)-l(D)}{\pi},
\end{equation*}
where $l(T)$ is the total edge length of $T$. Note that $2l(T)-l(D)$ is the sum of cell perimeters. The measure on
$\S{T}$ defined above can be normalized into a probability measure.
Below, that probability measure on $\S{T}$ is referred to as the
uniform probability measure on $\S{T}$ and a random split $\mathbf{S}$
of $T$ is said to be uniform if its distribution is the uniform
probability measure on $\S{T}$. Picking a split with a uniform
distribution can be done using a two-step procedure:
\begin{enumerate}
\item Select a cell $c$ of the T-tessellation $T$ with a probability
  proportional to its perimeter.
\item Select the line splitting $c$ according to the uniform and
  isotropic probability measure on the set of lines hitting $c$.
\end{enumerate}

\section{A completely random T-tessellation}
\label{sec:random}

Let us start with a formal definition of a random T-tessellation. The space $\cT$ is equipped with the standard hitting $\sigma$-algebra
$\sigma(\cT)$ (see \cite{matheron}) generated by events of the form
\begin{equation*}
  \left\{T \in \cT:\left(\bigcup_{e\in\sed{T}}e\right) \cap K \neq\emptyset\right\}
\end{equation*}
where $K$ runs through the set of compact subsets of $D$. \emph{A
  random T-tessellation}  is a random variable $\TT$ taking
values in $(\cT,\sigma(\cT))$.

Our candidate of completely random T-tessellation is based on the Poisson line process, denoted here by $\mathbf{L}$. The process $\mathbf{L}$ is supposed to have a unit linear intensity (mean length per unit area). Since only T-tessellations of the bounded domain $D$ are considered, only the restriction of $\mathbf{L}$ to $D$, also denoted $\mathbf{L}$ for sake of simplicity, is relevant. Consider the probability measure $\mu$ on $\sigma(\Te)$ defined by
\begin{equation}
  \label{eq:def:mu}
  \mu(A) = Z^{-1} \mean \sum_{T\in\cT(\mathbf{L})}\mathbf{1}_A(T), \quad A\in\sigma(\Te),
\end{equation}
where $Z$ is a normalizing constant. In order to prove that $\mu$ is well-defined, it must be checked that the mean number of T-tessellations in $\cT(\mathbf{L})$ is finite. The latter result is a consequence of the following more  general result.
\begin{theorem}
  \label{thm:exponential:moment}
  Let $\mathbf{L}$ be the unit Poisson line process restricted to the bounded domain $D$ and let $F$ be a stable functional on $\cT$. Then
  \begin{equation*}
    \mean \sum_{T\in\cT(\mathbf{L})}F(T)^x
  \end{equation*}
  is finite for any real $x$.
\end{theorem}
Theorem \ref{thm:exponential:moment} can be proved based of inequality \eqref{eq:tessellation:number} as shown in Appendix~\ref{appendix:exponential:moment}. Applying Theorem \ref{thm:exponential:moment} with $F\equiv 1$ proves that the measure $\mu$ defined by Equation \eqref{eq:def:mu} is indeed finite. The normalizing constant $Z$ has no known analytical expression. As a further consequence of Theorem \ref{thm:exponential:moment}, for $\TT\sim\mu$ and for any stable functional $F$, all moments of $F(\TT)$ are finite. 

If $\TT\sim\mu$, then given $L(\TT)=L$, $\TT$ is uniformly distributed on the finite set $\cT(L)$. It should be noticed that $L(\TT)$ is not Poisson. In particular, given the number of lines in $L(\TT)$, the probability of a configuration of lines is proportional to the size of $\cT(\mathbf{L})$, i.e.\ to the number of T-tessellations supported by the lines. Further insights into the distribution $\mu$ can be gained based on Papangelou kernels. 

Papangelou kernels were defined for point processes \cite{papangelou, kallenberg}. Heuristically, the (first-order) Papangelou kernel is the conditional probability to have a point at a given location given the point process outside this location. Formally, the Papangelou kernel is defined based on a decomposition of the reduced Campbell measure. The latter can be interpreted as the joint distribution of a typical point of a point process and the rest of the point process. 

Extending these notions to arbitrary random T-tessellations is not straightforward because tessellation components (vertices, edges) are more geometrically constrained than isolated points. In T-tessellations, features that can easily be added or removed are (non-blocking) segments. Below, segments that can be added to a T-tessellation $T$ will be identified to splits and $\S{T}$ will denote the space of such segments. Conversely, $\M{T}$ will denote the (finite) set of non-blocking segments that can be removed from $T$. Let $\CSplitSpace$ be the space
\begin{equation*}
  \CSplitSpace = \left\{(s,T):T\in\cT,s\in\S{T}\right\}.
\end{equation*}
\begin{definition} \label{def:split:campbell}
  The split (reduced) Campbell measure of a random T-tes\-sel\-la\-tion $\TT$ is defined by the following identity
  \begin{equation}
    \label{eq:def:split:campbell}
    \CSplitMeas(\phi) = \mean\sum_{m\in\M{\TT}}\phi\left(m,\TT\setminus\{m\}\right),\quad \phi:\CSplitSpace\rightarrow\R.
  \end{equation}
  The split Papangelou kernel is the kernel $\PSplit$ characterized by the identity
  \begin{equation}
    \label{eq:dev:split:papangelou}
    \CSplitMeas(\phi) = \mean\int_{\S{\TT}}\phi(s,\TT)\;\PSplit(\TT,ds),\quad \phi:\CSplitSpace\rightarrow\R.
  \end{equation}
\end{definition}
Note that for a point process, the reduced Campbell measure is defined by replacing the sum in Equation \eqref{eq:def:split:campbell} by a sum on all points of the process. For a T-tessellation, the sum is reduced to segments that are non-blocking. Continuing the comparison with point processes, it should be noticed that the Papangelou kernel $\PSplit(T,ds)$ is a measure on a space which depends on the realization $T$ while the Papangelou kernel of a planar point process is just a measure on $\R^2$.

A simple analytical expression of the split Papangelou kernel is available for the random T-tessellation $\TT\sim\mu$. Section~\ref{sec:space} introduced a uniform measure on splits. This measure with infinitesimal element denoted $dS$ induces a measure on non-blocking segments that can be added to a T-tessellation. Below an infinitesimal element of the latter measure is denoted by $ds$.
\begin{proposition} \label{prop:split:papangelou:mu}
  The split Papangelou kernel of a random T-tessellation $\TT\sim\mu$ can be expressed as
  \begin{equation}
    \label{eq:split:papangelou:mu}
    \PSplit(T,ds) = ds.
  \end{equation}
\end{proposition}
The proof of that proposition, detailed in Appendix~\ref{appendix:proof:split:papangelou}, is based on a
fundamental property of the Campbell measure for Poisson (line)
processes. Note that a Poisson line process is involved in the
definition of the measure $\mu$ defined on $\Te$, see
Equation~\eqref{eq:def:mu}. This result about the split Papangelou kernel is to be compared with the analog result concerning the Papangelou kernel of Poisson point processes. A heuristic interpretation of Proposition \ref{prop:split:papangelou:mu} is that the conditional distribution of a non-blocking segment of $\TT\sim\mu$ does not depend on the rest of $\TT$ (i.e.\ is uniform). As a direct consequence of Proposition \ref{prop:split:papangelou:mu}, we have the following Georgii-Nguyen-Zessin formula.
\begin{corollary} \label{cor:split:gnz:mu}
  If $\TT\sim\mu$, then the following identity holds
  \begin{equation}
    \label{eq:split:gnz:mu}
    \mean \sum_{m\in\M{\TT}}\phi\left(m,\TT\setminus\{m\}\right) =
    \mean \int_{\S{\TT}}\phi(s,\TT)\;ds,\quad \CSplitSpace:\phi\rightarrow\R.
  \end{equation}
\end{corollary}
Note that this result can be reformulated in terms of splits and merges. In particular, in the left-hand side of \eqref{eq:split:gnz:mu}, $\TT\setminus\{m\}$ can be written as $M\TT$ where $M$ is the merge consisting of removing the non-blocking segment $m$ from $\TT$. Furthermore, the removable segment $m$ is identified to the split $M^{-1}$. The reformulation is as follows: 
\begin{equation} \label{eq:alt:split:gnz:mu}
  \mean \sum_{M\in\M{\TT}}\phi\left(M^{-1},M\TT\right) =
    \mean \int_{\S{\TT}}\phi(S,\TT)\;dS.
\end{equation}

By analogy, a further result can be obtained for flips. The flip Campbell measure is defined on the space
\begin{equation*}
  \CFlipSpace = \left\{(F,T):T\in\cT,F\in\F{T}\right\}.
\end{equation*}
\begin{definition}
  \label{def:flip:campbell}
  The flip Campbell measure of a given random T-tes\-sel\-la\-tion $\TT$ is defined by the following identity
  \begin{equation}
    \label{eq:def:flip:campbell}
    \CFlipMeas(\phi) = \mean\sum_{F\in\F{\TT}}\phi\left(F^{-1},F\TT\right),\quad \phi:\CFlipSpace\rightarrow\R.
  \end{equation}
  The flip Papangelou kernel is the kernel $\PFlip$ characterized by the identity
  \begin{equation}
    \label{eq:dev:flip:papangelou}
    \CFlipMeas(\phi) = \mean\sum_{F\in\F{\TT}}\phi(F,\TT)\;\PFlip(\TT,F),\quad \phi:\CFlipSpace\rightarrow\R.
  \end{equation}
\end{definition}
Again a simple analytical expression of the flip Papangelou kernel can be derived for $\TT\sim\mu$. The derivation of this expression is based on the definition \eqref{eq:def:mu} involving a flat sum of T-tessellations supported by a line pattern. 
\begin{proposition} \label{prop:flip:papangelou:mu}
  The flip Papangelou kernel of a random T-tessellation $\TT\sim\mu$ can be expressed as
  \begin{equation}
    \label{eq:flip:papangelou:mu}
    \PFlip(T,F) \equiv 1.
  \end{equation}
\end{proposition}
A detailed proof is provided in Appendix~\ref{appendix:proof:flip:papangelou}.
Again, a heuristic interpretation of Proposition~\ref{prop:flip:papangelou:mu} is that all flips have the same conditional probability given the rest of  $\TT$. Proposition~\ref{prop:flip:papangelou:mu} leads to a Georgii-Nguyen-Zessin formula for flips.
\begin{corollary} \label{cor:flip:gnz:mu}
  If $\TT\sim\mu$, then the following identity holds
  \begin{equation}
    \label{eq:flip:gnz:mu}
    \mean \sum_{F\in\F{\TT}}\phi\left(F^{-1},F\TT\right) =
    \mean \sum_{F\in\F{\TT}}\phi(F,\TT),\quad \phi:\CFlipSpace\rightarrow\R.
  \end{equation}
\end{corollary}

Hence expressions of Papangelou kernels for splits and flips indicate that a random T-tessellation $\TT\sim\mu$ shows minimal spatial dependency, suggesting that $\mu$ can be considered as the distribution of a completely random T-tessellation. Below this model of completely random T-tessellation will be referred to as the CRTT model.

\section{Gibbsian T-tessellations}
\label{sec:gibbs:extension}

Although the completely random T-tessellation model introduced in the previous section shows appealing features, it may not be appropriate for representing real life structures which may exhibit some kind of regularity. This section is devoted to Gibbsian extensions allowing to control a large spectrum of T-tessellation features. Gibbs random T-tessellations are defined as follows.
\begin{definition}
  \label{def:gibbs}
  Let $h$ be a stable non-negative functional on $\cT$. The Gibbs random T-tessellation with unnormalized density $h$ is the random T-tessellation with distribution
  \begin{equation}
    \label{eq:def:extended:model}
    \PP(dT) \propto h(T)\; \mu (dT),
  \end{equation}
  where the proportionality constant is defined such that $\PP$ has total mass equal to $1$.
\end{definition}
Equation~\eqref{eq:def:extended:model} defines a probabilistic measure on $\cT$ if and only if the unnormalized density $h$ has a non-null finite integral with respect to $\mu$. The latter condition is guaranteed by the assumed stability of $h$. Note that for point processes (in bounded domains) too, the integrability of unnormalized densities with respect to the Poisson distribution is guaranted by a stability condition, see e.g. \cite[Section 3.7]{geyer:_likel} or \cite[Chapter 3]{ruelle69:_statis_mechan}. Following the Gibbsian terminology, $-\log h$ is referred to as the energy function.

Extending results of Section~\ref{sec:random}, the split and flip Papangelou kernels can be derived for Gibbs T-tessellations. As for Gibbs point processes, such results are obtained under a kind of hereditary condition.
\begin{definition}
  \label{def:heredity}
  A Gibbs random T-tessellation with unnormalized density $h$ is hereditary if it fulfills the following two condititions :
  \begin{itemize}
  \item For any pair $(T,S)$, $T\in\cT$ and $S\in\S{T}$, $h(ST)=0\Rightarrow h(T)=0$.
  \item For any pair $(T,F)$, $T\in\cT$ and $F\in\F{T}$, $h(FT)=0\Rightarrow h(T)=0$.
  \end{itemize}
\end{definition}
\begin{proposition} \label{prop:gnz:gibbs}
  If $\TT$ is a Gibbs random T-tessellation with hereditary unnormalized density $h$, then its split and flip Papangelou kernels can be expressed as follows
  \begin{equation}
    \label{eq:papangelou:gibbs}
    \PSplit(T,ds) = \frac{h\left(T\cup\{s\}\right)}{h(T)}\;ds,\quad
    \PFlip(T,F) = \frac{h\left(FT\right)}{h(T)}.
  \end{equation}
  And the following two Georgii-Nguyen-Zessin formulae hold
  \begin{eqnarray}
    \label{eq:gnz:split}
    \mean\sum_{m\in\M{\TT}}\phi\left(m,\TT\setminus\{m\}\right) & = &
    \mean\int_{s\in\S{\TT}}\phi(s,\TT)\frac{h\left(\TT\cup\{s\}\right)}{h(\TT)}\;ds,\\
    \label{eq:gnz:flip}
    \mean\sum_{F\in\F{\TT}}\phi(F^{-1},F\TT) & = &
     \mean\sum_{F\in\F{\TT}}\phi(F,\TT)\frac{h\left(F\TT\right)}{h(\TT)},
  \end{eqnarray}
  taking 0/0=0.
\end{proposition}
Proposition~\ref{prop:gnz:gibbs} is a straightforward consequence of propositions \ref{prop:split:papangelou:mu}--\ref{prop:flip:papangelou:mu} and corollaries \ref{cor:split:gnz:mu}--\ref{cor:flip:gnz:mu}.

\begin{example} 
  \label{ex:crt2}
  A first simple example is obtained with an energy function
  \begin{equation}
    \label{eq:crt2:density:scaling}
    -\log h(T) = -\nseint{T}\log\tau,
  \end{equation}
  which depends only on the number of internal segments (i.e.\ the number of supporting lines). The parameter $\tau$ may be tuned in order to control the number of cells. For the energy function \eqref{eq:crt2:density:scaling}, the Papangelou kernels are given by
  \begin{equation*}
    \PSplit(T,ds) = \tau\;ds,\quad
    \PFlip(T,F) \equiv 1.
  \end{equation*}
Note that the Papangelou kernels do not depend on $T$. Therefore the model defined by the energy \eqref{eq:crt2:density:scaling} behaves like the completely random T-tessellation introduced in Section~\ref{sec:random} up to the scaling parameter $\tau$. Below this model is referred to as the completely random T-tessellation with scaling parameter $\tau$.
A realization of such a model is shown in Figure~\ref{fig:ttessel:crt2}. It was generated using the Metropolis-Hastings-Green algorithm further described in Section~\ref{sec:simulation}. 
\end{example}
\begin{example} 
  \label{ex:acs}
  The energy function
  \begin{equation}
    \label{eq:acs:density}
    -\log h(T) = \frac{\tau}{\pi}l(T)+\nveint{T}\log 2-\nseint{T}\log\tau
  \end{equation}
  yields a model of T-tessellation proposed by Arak et al.\
  \cite{ArakClifSurg93}. In their paper, Arak et al.\ introduced a
  general probabilistic model for planar random graphs. This model
  involves 2 scalar and 3 functional parameters. Depending on these
  parameters, the model may yield segment patterns, polyline patterns,
  polygons or T-tessellations. A realization of the
  Arak-Clifford-Surgailis model is shown in
  Figure~\ref{fig:ttessel:acs}. The parameter $\tau$ was adjusted in order to yield T-tessellations with approximately the same mean number of cells as the CRTT model shown in Figure~\ref{fig:ttessel:crt2}.

  For the energy \eqref{eq:acs:density},
  the normalizing constant (partition function) has a simple
  analytical expression. The parameter $\tau$ is a scaling
  parameter. It can be checked that the linear intensity of this
  random T-tessellation is equal to $\tau$. Furthermore, the model has
  a Markov spatial property in the sense that the conditional
  distribution of $\TT$ on a region $B$ given the complement $D\setminus B$ depends on the complement only through its
  intersection with the boundary of $B$. Also Arak et al.\ provided a
  one-pass algorithm for simulating their model. Miles and
  Mackisack~\cite{ref/947} have derived the distributions of edge and
  segment lengths and of cell areas.

  The split and flip Papangelou kernels of the Arak-Clifford-Surgailis model can be computed from equations \eqref{eq:papangelou:gibbs} and \eqref{eq:acs:density}. In particular, for a splitting segment $s$ with no end lying on the boundary of the domain $D$, the split Papangelou kernel has a log-density equal to
  \begin{equation*}
    \log\frac{\PSplit(T,ds)}{ds} = -\frac{\tau}{\pi}l(s)-2\log 2+\log\tau,
  \end{equation*}
  where $l(s)$ is the segment length. Compared to the CRTT model, splits by short (non-blocking) segments are favoured. A splitting segment is short if the splitted cell is small or if it lies at the periphery of large cells. Therefore one may expect ACS tessellations to show more small cells than CRTT. This was checked numerically. Large samples of realizations of both the CRTT and the ACS model were generated. Cell area distributions were plotted as Lorenz curves (fraction of smallest cells versus area fraction of the covered domain), see Figure \ref{fig:stats}. Plots confirmed that the ACS model tends to yield more small cells than the CRTT model. 
\end{example}
\begin{example} \label{ex:areas} 
In order to obtain T-tessellations with less dispersed cell areas, one may consider the following energy function
\begin{equation}
  \label{eq:energy:areas}
  -\log h(T) = - \nseint{T}\log\tau + \alpha a^2(T),
\end{equation}
where $a^2(T)$ is the sum of squared cell areas for the tessellation
$T$. The statistic $a^2(T)$ is minimal for equal size cells. Furthermore note that since $a^2(T)\leq a(D)^2$, the energy \eqref{eq:energy:areas} is stable. Therefore, the probabilistic measure \eqref{eq:def:extended:model} is well-defined for energy \eqref{eq:energy:areas}. A realization of this model is shown in Figure~\ref{fig:ttessel:areas}. The values of $\tau$ and $\alpha$ were adjusted in order to produce approximately the same mean number of cells as in figures \ref{fig:ttessel:crt2} and \ref{fig:ttessel:acs}. The Lorenz curve of Figure \ref{fig:stats} confirmed that the area penalty in \eqref{eq:energy:areas} yields much less small cells than under the CRTT model.
\end{example}
\begin{example}
  \label{ex:angles}
  In order to control angles between incident segments, one may consider the following energy
  \begin{equation}
    \label{eq:energy:angles}
    -\log h(T) = - \nseint{T}\log\tau - \beta \sum_{v \in\sve{T}} \phi(v),
  \end{equation}
  where $\phi(v)$ denotes the smallest angle between two edges incident to $v$ ($\phi(v)$ is close to $\pi/2$ if the segments meeting at $v$ are almost perpendicular). Since the angle $\phi(v)$ is bounded, the energy~\eqref{eq:energy:angles} is stable. A realization of this model is shown in Figure~\ref{fig:ttessel:angles}. Histograms of angles measured from simululated tessellations are shown in Figure \ref{fig:stats}. The angle distribution is much more concentrated towards $\pi/2$ than under the CRTT model.
\end{example}
\begin{figure}
  \centering
  \subfloat[\label{fig:ttessel:crt2}CRTT.  \protect$\tau=1.9$.]{\includegraphics[width=4cm]{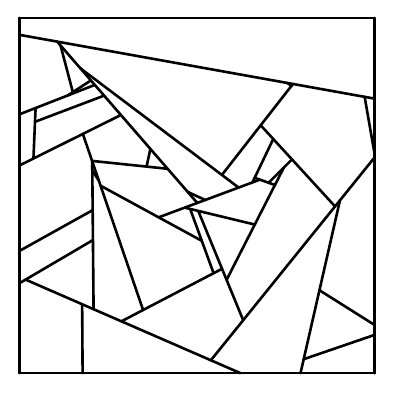}}\hspace{1cm}
  \subfloat[\label{fig:ttessel:acs}ACS. \protect$\tau=10.75$.]{\includegraphics[width=4cm]{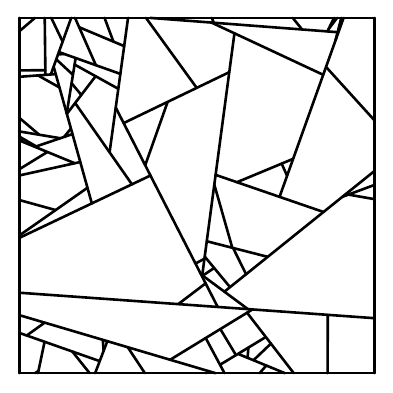}}\\
  \subfloat[\label{fig:ttessel:areas}Area penalty. \protect$\tau=0.043$, \protect$\alpha=\numprint{10000}$.]{\includegraphics[width=4cm]{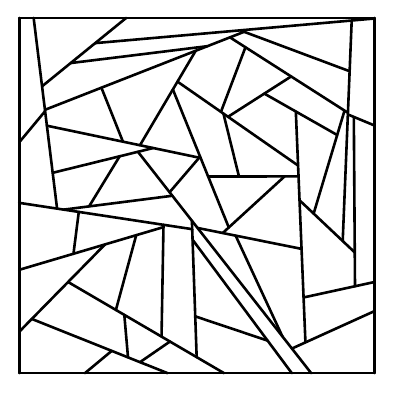}}\hspace{1cm}
  \subfloat[\label{fig:ttessel:angles}Angle penalty, \protect$\tau=12.1$, \protect$\beta=2.5$.]{\includegraphics[width=4cm]{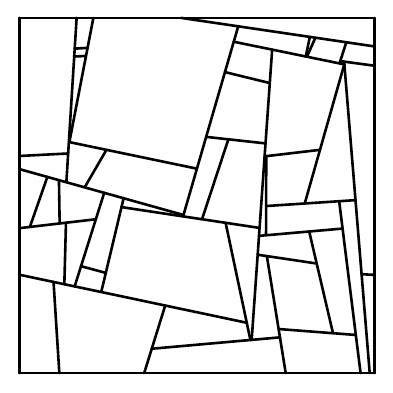}}
  \caption{Realizations of four model examples.}
  \label{fig:realization:alpha:beta}
\end{figure}
\begin{figure}
  \includegraphics{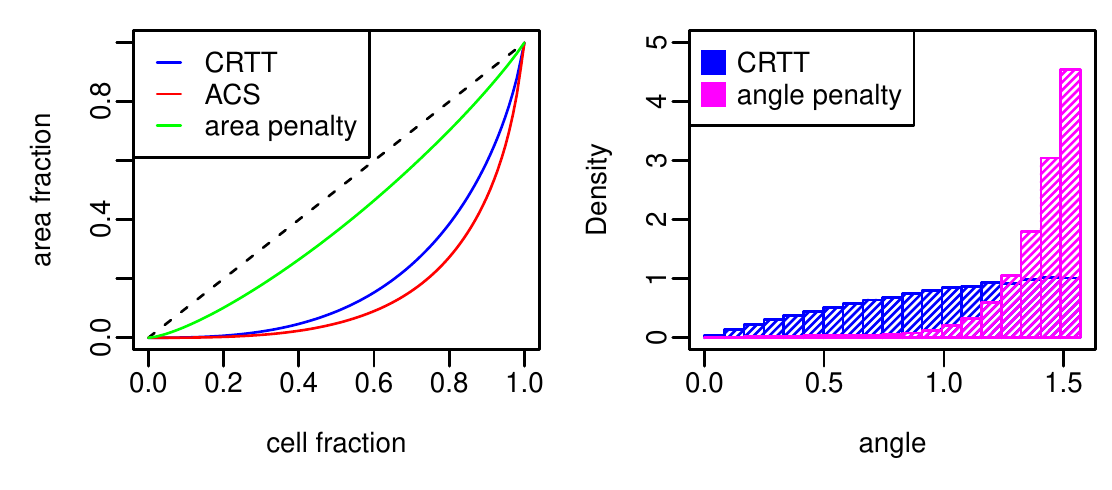}
  \caption{Statistical comparisons between three Gibbs model examples and the CRTT model. Left: distributions of cell areas of the CRTT, ACS and squared area penalty models shown as Lorenz curves (\protect$x$-axis : fraction of smallest cells, \protect$y$-axis : covered area fraction, discontinuous : theoretical curve for cells with uniformly distributed areas). Right : distribution  of acute angles between incident segments of the CRTT and angle penalty model.}
  \label{fig:stats}
\end{figure}

\section{A Metropolis-Hastings-Green simulation algorithm}
\label{sec:simulation}

In this section, we derive a simulation algorithm for random Gibbs
 T-tessellations. This algorithm is a
special case of the ubiquitous Metropolis-Hastings-Green algorithm, see
e.g.\ \cite{ref/156, ref/1113, geyer:_likel}. It consists in
designing a Markov chain with state space $\Te$ and with invariant
distribution the target probability measure $P$.

The design of a Metropolis-Hastings-Green algorithm involves two basic ingredients: random proposals of updates and rules for accepting or rejecting updates. Here, three types of updates are considered: splits, merges and flips. Proposition kernels based on these updates are not absolutely continuous with respect to the reference measure $\mu$. Thus like the birth-death Metropolis-Hastings algorithm used to simulate Gibbs point processes, the simulation algorithm described below is an instance of Green extension of Metropolis-Hastings algorithm. Rules for accepting updates are based on Hastings ratios. The computation of Hastings ratios involves ratios of densities with respect to symmetric measures on pairs of T-tessellations which differ by a single update.

\subsection{Symmetric measures on pairs of tessellations}
\label{sec:smf:sym:measures}

The space of pairs of T-tessellations of the domain $D$ which differ
only by a single split, merge or flip is denoted $\Te_\text{smf}^2$.
That space can be decomposed into two subspaces $\Te_\text{sm}^2$
(pairs which differ by a split or a merge) and $\Te_\text{f}^2$ (pairs
which differ by a flip). 
As a preliminary to the design of the
simulation algorithm, symmetric measures on $\Te_\text{sm}^2$ and on
$\Te_\text{f}^2$ are required. A measure $\nu$ on say $\Te_\text{sm}^2$ is
symmetric if for any measurable subset $A\subset\Te_\text{sm}^2$, the
following equality holds
\begin{equation*}
  \int_{\Te_\text{sm}^2} \mathbf{1}_A(T_1,T_2)\;\nu(dT_1,dT_2) =
  \int_{\Te_\text{sm}^2} \mathbf{1}_A(T_2,T_1)\;\nu(dT_1,dT_2).
\end{equation*}

The following result is a direct consequence of the Georgii-Nguyen-Zessin formulae for the completely random T-tessellation.
\begin{proposition} \label{prop:symmetric:measures}
Let $\TT\sim \mu$. The measures $\mu_\text{sm}$ on  $\Te_\text{sm}^2$ and $\mu_\text{f}$ on  $\Te_\text{f}^2$ defined by
\begin{equation}
  \label{eq:def:sym:meas:sm}
  \mu_\text{sm}(A) = \mean \int_{\mathcal{S}_{\TT}} 
  \mathbf{1}_A(\TT,S\TT)\;dS  + \mean \sum_{M\in\mathcal{M}_{\TT}} 
  \mathbf{1}_A(\TT,M\TT),
  \quad A\subset\Te_\text{sm}^2,
\end{equation}
and 
\begin{equation}
  \label{eq:def:sym:meas:f}
  \mu_\text{f}(A) = \mean\sum_{F\in\mathcal{F}_{\TT}} 
  \mathbf{1}_A(\TT,F\TT),
  \quad A\subset\Te_\text{f}^2,
\end{equation}
are symmetric.
\end{proposition}
A detailed proof is given in Appendix \ref{appendix:proof:prop:symmetric:measures}.

\subsection{The algorithm}
\label{sec:algorithm}

The proposed algorithm proceeds by successive updates. It involves two proposition kernels. The first kernel is based on splits and merges. It is quite similar to births and deaths used for simulating point processes. Consider two non-negative
functions $p_\text{s}$ and $p_\text{m}$ defined on $\Te$ such that $p_\text{s}+p_\text{m}\leq 1$. For any T-tessellation $T$, $p_\text{s}(T)$ (resp.\ $p_\text{m}(T)$) is the probability for considering applying a split (resp.\ merge) to $T$. If $p_\text{s}(T)+p_\text{m}(T)<1$, with probability $1-p_\text{s}(T)-p_\text{m}(T)$, no update is considered.

The rules for selecting an update of a given type are determined by two families of non-negative functions $q_\text{s}$ and $q_\text{m}$. The function $q_\text{s}$ is defined on the set of pairs $(T,S)$ where $T\in\Te$ and $S\in\S{T}$. For any T-tessellation $T$, $q_\text{s}(T,.)$ is supposed to be a probability density with respect to the uniform measure on $\S{T}$ defined by Equation~\eqref{eq:def:dS}. Concerning merges, $q_\text{m}(T,.)$ is just a discrete distribution (finite family of probabilities) defined on $\M{T}$.

Consider the following proposition kernel on $\cT$:
\begin{equation*}
  Q_{\text{sm}}(T,d\tilde{T})  =  \left\{
  \begin{array}{lll}
    p_{\st}(T) q_{\st}(T,ST)\;dS & \text{if }\tilde{T}=ST, & S\in\S{T},\\
    p_{\mt}(T) q_{\mt}(T,MT) & \text{if }\tilde{T}=MT, & M\in\M{T}.
  \end{array}
  \right.
\end{equation*}
Acceptation of an update is based on the following Hastings ratio
\begin{equation*}
  \frac{h(\tilde{T})}{h(T)} \frac{q_{\text{sm}}(\tilde{T},T)}{q_{\text{sm}}(T,\tilde{T})},
\end{equation*}
where $q_{\text{sm}}(\tilde{T},T)$ is the density of $Q_{\text{sm}}\mu$ with respect to the symmetric measure $\mu_{\text{sm}}$. It can be easily checked that
\begin{equation*}
  q_{\text{sm}}(T,d\tilde{T})  =  \left\{
  \begin{array}{lll}
    p_{\st}(T) q_{\st}(T,ST) & \text{if }\tilde{T}=ST, & S\in\S{T},\\
    p_{\mt}(T) q_{\mt}(T,MT) & \text{if }\tilde{T}=MT, & M\in\M{T}.
  \end{array}
  \right.
\end{equation*}
Therefore the Hastings ratio for splits and merges can be expressed as 
\begin{equation}
  \label{eq:def:hastings}
  r_\text{t}(T,U) =
  \frac{h(UT)}{h(T)}
  \frac{p_{t^{-1}}(UT)}{p_t(T)}
  \frac{q_{t^{-1}}(UT,U^{-1})}{q_t(T,U)},
\end{equation}
where $t\in\{\text{s},\text{m}\}$, $\text{s}^{-1}=\text{m}$,
$\text{m}^{-1}=\text{s}$, $U\in\S{T}$ if $t=\text{s}$ and
$U\in\M{T}$ if $t=\text{m}$.

Using splits and merges, one would explore only a subspace of $\cT$: nested T-tessellations. Hence, updates based on flips are required. Let $p_\ft=1-p_\st-p_\mt$ be the probability of trying to flip the current tessellation. And, for any $T\in\cT$, let $q_\ft(T,.)$ be a finite distribution defined on $\F{T}$. The pair $(p_\ft,q_\ft)$ defines a proposition kernel $Q_\ft$. It is easy to check that $Q_\ft\mu$ is absolutely continuous with respect to the symmetric measure $\mu_{\text{f}}$ with density $p_\ft q_\ft$. Furthermore, the Hastings ratio for flips can be expressed as in Equation~\eqref{eq:def:hastings} where $\text{t}=\ft$, $\ft^{-1}=\ft$ and $U\in\F{T}$.

The simulation algorithm requires as inputs: the unnormalized density
$h$ of the target distribution, an initial T-tessellation of the
domain $D$, the triplet
$(p_\text{s},p_\text{m},p_\text{f})$ and the triplet
$(q_\text{s},q_\text{m},q_\text{f})$.

Each iteration consists of the following steps:
\begin{enumerate}
\item The current T-tessellation is $T_n$.
\item Draw at random the update type $t$ from
  $\{\text{s},\text{m},\text{f}\}$ with probabilities
  $p_\text{s}(T_n)$, $p_\text{m}(T_n)$ and $p_\text{f}(T_n)$.
\item If there is no update of type $t$ applicable to $T_n$, $T_{n+1}=T_n$. 
\item Draw at random an update $U$ of type $t$ according to the
  distribution defined by $q_t(T_n,.)$.
\item Compute the Hastings ratio $r_{\text{t}}(T_n,U)$.
\item Accept update ($T_{n+1}=UT_n$) with probability
  \begin{equation*}
    \min\{1,r_{\text{t}}(T_n,U)\}.
  \end{equation*}
\end{enumerate}

Let $\mathbf{T}_n$ be the Markov chain defined by the algorithm
above. By construction $\mathbf{T}_n$ is reversible with respect to
the target distribution $P$. Further properties of $\mathbf{T}_n$ are
discussed in Section \ref{sec:simulation:convergence}.

\begin{example}
  \label{ex:simple:smf}
  A simple version of the algorithm is obtained by choosing
  probabilities $p_\st$, $p_\mt$ and $p_\ft$ not depending on $T$ and
  by using uniform random proposals of updates. Define
  \begin{eqnarray}
    \label{eq:q:split:unif}
    q_\text{s}(T,S) & = & \frac{\pi}{2l(T)-l(D)},\\
    \label{eq:q:merge:unif}
    q_\text{m}(T,M) & = & \frac{1}{\nnbseint{T}},\\
    \label{eq:q:flip:unif}
    q_\text{f}(T,F) & = & \frac{1}{2 \nbseint{T}},
  \end{eqnarray}
  in order to get uniformly distributed splits, merges and flips. For
  a split $S$, the Hastings ratio is
  \begin{equation*}
    r_\text{s}(T,S)
    =
    \frac{h(ST)}{h(T)}\frac{p_\mt}{p_\st}
    \frac{2l(T)-l(D)}{\pi(\nnbseint{T}+1-\xi)}
  \end{equation*}
  where $\xi\in\{0,1,2\}$ is the number of internal non-blocking
  segments of $T$ incident to the ends of the splitting
  segment. Similary, for a merge the Hastings ratio is
  \begin{equation*}
    r_\text{m}(T,M)
    =
    \frac{h(MT)}{h(T)}\frac{p_\st}{p_\mt}
    \frac{\pi \nnbseint{T}}{2(l(T)-l(M))-l(D)},
  \end{equation*}
  where $l(M)$ is the length of the edge removed by $M$. For a
  flip $F$, the Hastings ratio is
  \begin{equation*}
    r_\text{f}(T,M)
    =
    \frac{h(FT)}{h(T)}
    \frac{\nbseint{T}}{\nbseint{FT}},
  \end{equation*}
  where $\nbseint{FT}-\nbseint{T}$ varies from $-2$ up to $2$. 

  The computation of Hastings ratios requires to keep track of the
  numbers of blocking and non-blocking segments, the total length of
  internal edges. Also, one needs to predict the variations of the
  non-blocking segment number induced by a split and of the blocking
  segment number induced by a flip. Finally one must also predict the
  variation of the unnormalized density $h$ induced by any of the three
  updates.
\end{example}

Simulations based on Metropolis-Hastings-Green algorithms are commonly monitored based on plots as those shown in Figure~\ref{fig:simul:monitoring}. The CRTT model and the model of Example~\ref{ex:areas} were considered. Additionally, another model combining both penalizations of examples \ref{ex:areas} and \ref{ex:angles} was also considered. The values of parameters $\alpha$ and $\beta$ were chosen in order to produce rather rectangular cells. A realization of the latter model is shown in Figure \ref{fig:ttessel:landscape}.
\begin{figure}
  \centering
  \includegraphics[width=4cm]{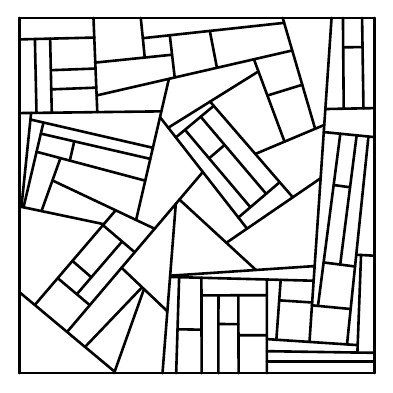}
  \caption{A realization of the model obtained by combining energy terms involved in example models \protect\ref{ex:areas} and \protect\ref{ex:angles}. Parameter values: \protect$\tau=2.0$, \protect$\alpha=\numprint{93000}$, \protect$\beta=200$.}
  \label{fig:ttessel:landscape}
\end{figure}
\begin{figure}
  \centering
  \subfloat[\label{fig:monitor:crt2}CRTT]{\includegraphics{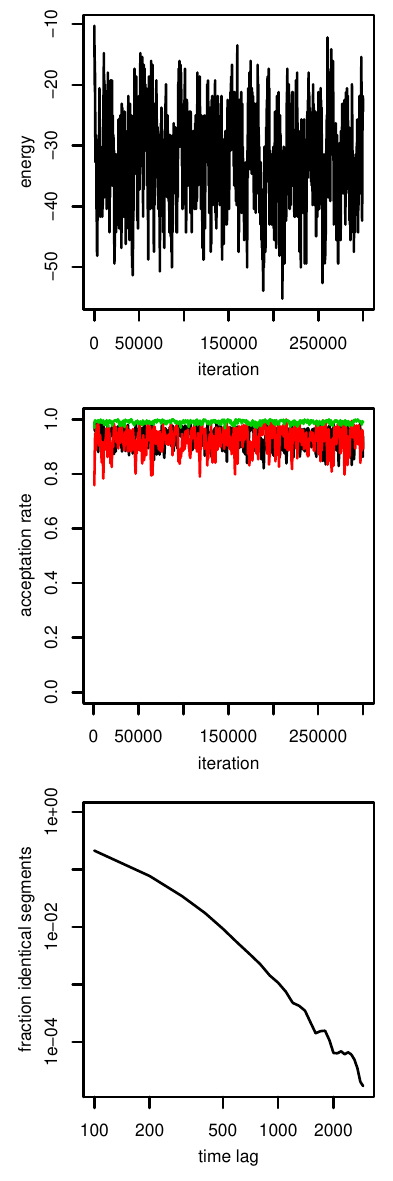}}
  \subfloat[\label{fig:monitor:areas}Model of Example \protect\ref{ex:areas}]{\includegraphics{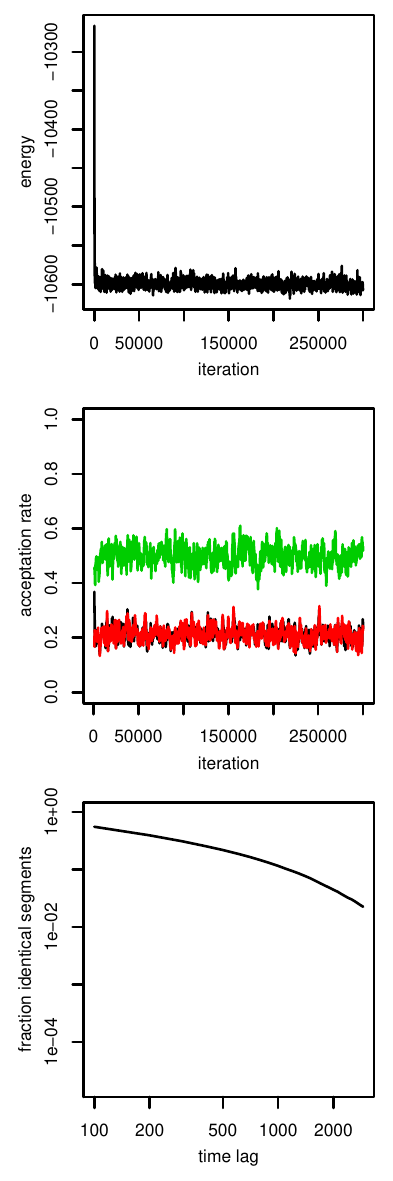}}
  \subfloat[\label{fig:monitor:landscape}Example \protect\ref{ex:areas}+\protect\ref{ex:angles}]{\includegraphics{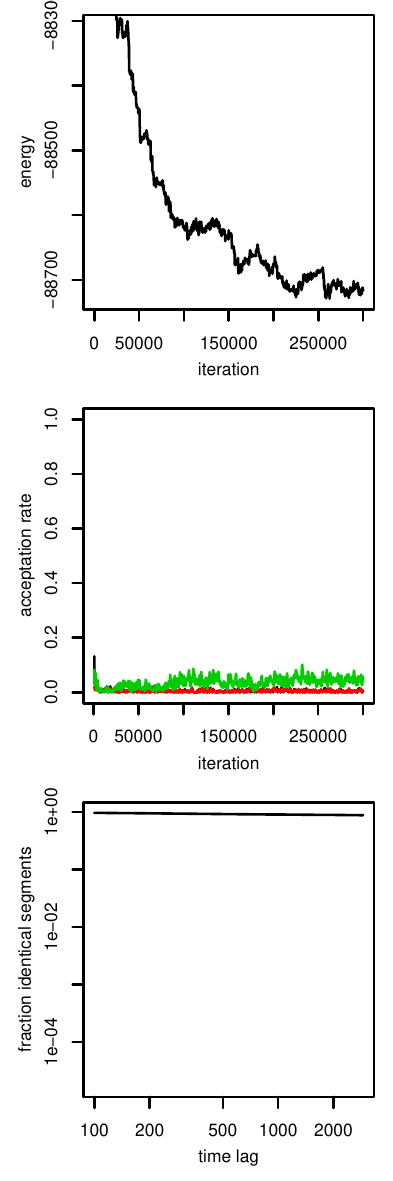}}
  \caption{Plots for monitoring simulations. Top row: time-evolution of the energy. Middle row: time-evolution of acceptation rates (splits in black, merges in red, flips in green). Bottom row: fraction of common segments between two tessellations versus time lag (log-log scale).}
  \label{fig:simul:monitoring}
\end{figure}

As a first monitoring tool, one can follow the time-evolution of the energy. A convenient starting initial tessellation to start with is the empty tessellation. The empty tessellation may have quite a high energy. Obviously the Markov chain should not be sampled before the energy stabilizes at its equilibrium level. The time needed for burnin is heavily dependent on the form of the energy function. As shown by Figure~\ref{fig:simul:monitoring}, burnin is very short for the CRTT model: less than \numprint{1000} iterations. Burnin is also short for the model with penalty on cell area variability (Example~\ref{ex:areas}), although the initial empty tessellation shows a very high energy compared to the energy value at equilibrium. For a model with very severe penalties like the one combining penalties of examples \ref{ex:areas} and \ref{ex:angles} (realization shown in Figure~\ref{fig:ttessel:landscape}), burnin takes several hundreds of thousands iterations. 

Acceptation rates are also quite informative. As expected, the CRTT model is easy to simulate with very high acceptation rates. More constrained models show lower acceptation rates. In model of Example \ref{ex:areas}, unbalanced splits generating small cells or merges creating large cells are often rejected. For an extreme model such as the one combining examples \ref{ex:areas} and \ref{ex:angles}, the acceptation rates drop down at rather low values. In such a case, it might be worthwhile to consider more state dependent proposals than those of Example \ref{ex:simple:smf}. For instance, one could try to propose more central splits.

In general, the simulation Markov chain shows time-correlation: two successive T-tessellations differ only locally. Therefore, the Markov chain is commonly subsampled. The sampling period to be used depends on the range of time-correlations. As an empirical method for assessing time correlations, one may consider the percentage of segments left unchanged at different time lags. Plots from Figure~\ref{fig:simul:monitoring} show contrasted situations. For the CRTT model, after \numprint{500} iterations, about only \numprint[\%]{1} of segments are left unchanged.  For the model of Example~\ref{ex:areas}, such an updating rate requires more than \numprint{2000} iterations. The last model shows a very poor updating: very long runs are needed in order to obtain uncorrelated realizations.

\subsection{Convergence}
\label{sec:simulation:convergence}

By construction, see \cite{ref/1113} or \cite{geyer:_likel} for more
details, the Markov chain $\mathbf{T}_n$ defined in Section
\ref{sec:algorithm} is reversible with respect to the target
distribution $P\propto h\mu$. In order to get a convergence result,
one needs to establish that $\mathbf{T}_n$ is also irreducible and
aperiodic, see e.g.\ \cite{ref/931, ref/938}.

\begin{proposition}
  \label{prop:irreducibility}
  Under both conditions given below, the Markov chain $\mathbf{T}_n$
  is irreducible.
  \begin{enumerate}
  \item For any T-tessellation $T$ that can be merged, there exists a
    merge $M$ such that
    \begin{equation}
      \label{eq:cond:1:irreducibility}
      p_{\mt}(T)q_{\mt}(T,M)>0 \text{ and } h(MT)p_{\st}(MT)q_{\st}(MT,M^{-1})>0.
    \end{equation}
  \item For any T-tessellation $T$ that can be flipped and any flip
    $F\in\F{T}$,
    \begin{equation}
      \label{eq:cond:2:irreducibility}
      h(T)p_{\ft}(T)q_{\ft}(T,F)>0.
    \end{equation}
  \end{enumerate}
\end{proposition}
The detailed proof given in Appendix \ref{appendix:proof:prop:irreducibility} follows the same guidelines as the proof used for point processes, see e.g. \cite{geyer:_likel, ref/156}.

Note that the first condition in Proposition~\ref{prop:irreducibility}
is similar to the irreducibility condition involved in the birth and
death Metropolis-Hastings algorithm for the simulation of Gibbs point
processes. The second condition is somewhat stronger than the first
one since it must hold for all flips. In particular the density $h$
must keep strictly positive.

For the simple version of the simulation algorithm described in Example~\ref{ex:simple:smf}, both conditions of
Proposition~\ref{prop:irreducibility} hold if and only if the target
density $h$ is strictly positive.
\begin{proposition}
  \label{prop:aperiodicity}
  If either $p_\ft(T_D)>0$ or $p_\mt(T_D)>0$, the Markov chain
  $\mathbf{T}_n$ is aperiodic.
\end{proposition}
\begin{proof}
  One needs to prove that there exists a T-tessellation $T$ such that
  the conditional probability that $\mathbf{T}_{n+1}=T$ given that
  $\mathbf{T}_n=T$ is positive. Consider the empty tessellation $T_D$.
  If $p_\ft(T_D)>0$, a flip of $T_D$ is considered with a positive
  probability. As there is no flip applicable to $T_D$,
  $\mathbf{T}_{n+1}=\mathbf{T}_n=T_D$. The same argument holds if
  $p_\mt(T_D)>0$.
\end{proof}

Since the Markov chain $\mathbf{T}_n$ is reversible with respect to
$P\propto  h\mu$ and in view of
propositions~\ref{prop:irreducibility} and \ref{prop:aperiodicity}, we
have the following convergence result:
\begin{proposition}
  \label{prop:convergence}
  Let $P$ be a distribution on $\Te$ with
  unnormalized density $h$ with respect to $\mu$. Under the
  conditions of propositions~\ref{prop:irreducibility} and
  \ref{prop:aperiodicity}, for $P$-almost any T-tessellation $T$ the
  conditional distribution of the Markov chain $\mathbf{T}_n$ given
  that $\mathbf{T}_0=T$ converges to $P$ in total variation.
\end{proposition}

\section{Discussion}

The main feature of the completely random T-tessellation introduced in this paper is that both split and flip Papangelou kernels have very simple expressions showing a kind of lack of spatial dependency. It would be of high interest to further investigate this model. Since analytical probabilistic results are available for the Arak-Clifford-Surgailis (Example~\ref{ex:acs}), it is expected that such results could also be derived for our model. In particular, the following issues are of interest:
\begin{itemize}
\item Is there a simple expression for the normalizing constant $Z$ involved in Equation~\eqref{eq:def:mu}?
\item Does our model share the same Markovian property as the Arak-Clifford-Surgailis model?
\item Is it possible to derive an exact non-iterative simulation algorithm similar to the one proposed to Arak et al.\ \cite{ArakClifSurg93}?
\end{itemize}

It should be noticed that the choice of a reference measure for a class of Gibbs models is in some way a matter of convention. Since the CRTT model introduced in this paper is absolutely continuous with respect to the Arak-Clifford-Surgailis, one could use the ACS model as a reference measure instead. In particular, such a substitution would have a slight impact on the simulation algorithm described in Section~\ref{sec:simulation}: only the Hastings ratio should be modified by an extra multiplicative term.

In this paper, we focused on T-tessellations of a bounded domain. We expect the completely random T-tessellation model to be extensible to the whole plane. Concerning Gibbs variations, extensions to the whole plane must involve some extra conditions on the unnormalized density.

Another model for random T-tessellations is the STIT tessellation introduced by Nagel and Wei\ss\ \cite{nagel-2003, nagel-2005}. We conjecture that the STIT model can be seen as a Gibbs T-tessellation. However the derivation of its density with respect to the CRTT distribution remains an open problem. Obviously our simulation algorithm is not appropriate for simulating a STIT tessellation. Realizations of a STIT are obtained by series of splits (starting from the empty tessellation) generating nested T-tessellations. Modifying a nested tessellation by a flip may yield a non-nested tessellation. Even a modification of our algorithm consisting only in splits and merges would be inefficient. Given a nested T-tessellation, the only way to modify the first splitting segment (running across the whole domain) requires to remove all subsequent segments in reverse order. Fundamentally, STIT tessellations and the Gibbs T-tessellations considered here have different geometries.

A simulation algorithm of Metropolis-Hastings-Green type has been devised. This algorithm is based on three local modifications: split, merge and flip. Conditions ensuring the convergence (in total variation) of the simulation Markov chain were derived. Also geometric ergodicity of the Markov chain has to be investigated. Under geometric ergodicity and extra conditions (e.g.\ Lyapunov condition), central limit theorems can be obtained for averages of functionals based on Monte-Carlo samples of the Markov chain $\mathbf{T}_n$. In particular, geometric ergodicity has been proved for the Metropolis-Hastings-Green algorithm devised for simulating Gibbs point processes, see \cite{ref/156, geyer:_likel}. We tried to follow a similar approach without success. The main difficulty is related to the fact that, when one tries to empty a given T-tessellation, the number of operations (merges and flips) involved is not easily bounded.

The pioneering work of Arak, Clifford and Surgailis gave rise to developments by Schreiber, Kluszczy\'{n}ski and van Lieshout \cite{schreiber-2005, klu-2007, schreiber-2010} on polygonal Markov fields. Although the geometry of those polygonal fields differs from T-tessellations, it is of interest to notice that the simulation procedures proposed by these authors involve non-local updates based on so-called disagreement loops. It seems worthwhile to investigate whether similar non-local updates could be adapted for T-tessellations.

As mentioned in Example~\ref{ex:acs}, the Arak-Clifford-Surgailis model is able to produce planar graphs not only with T-vertices but also other types of vertices (X, I, L and Y). It remains an open problem how to extend the framework discussed here to a more general graph. Obviously, other operators than splits, merges and flips (as defined here) are required. 

A practical important issue is the inference of the model parameters. Since the likelihood is known up to an untractable normalizing constant and simulations can be done by a Metropolis-Hastings-Green algorithm, Monte-Carlo maximum likelihood \cite{geyer:_likel} can be considered. Another approach would consist into deriving a kind of pseudolikelihood similar to the one developped for point processes, see \cite{jensen, billiot}. As a starting point, one could use the Georgii-Nguyen-Zessin formulae for Gibbs T-tessellations derived in Proposition~\ref{prop:gnz:gibbs}.

\paragraph{Acknowledgements} \emph{This work has been presented in several
workshops such as the 6th and 7th French-Danish Workshops on Spatial
Statistics and Image Analysis for Biology (Skagen 2006 and Toulouse
2008), the 14th Workshop on Stochastic Geometry, Stereology and Image
Analysis (Neudietendorf 2007) and the first Journées de géométrie
aléatoire (Lille 2008). Our deepest gratitude to our colleagues (INRA-INRIA Payote network, \emph{Géométrie stochastique}  research group of University Lille 1,...) who
showed interest in our work. Special thanks to Jonas Kahn who
met the challenge of proving that the measure defining our completely
random T-tessellation is indeed finite.}  

\bibliographystyle{plain}
\bibliography{tess}

\appendix

\section{Proof of Theorem \ref{thm:exponential:moment}}
\label{appendix:exponential:moment}

Let $\TT\sim\mu$ and $F$ be a stable functional on $\cT$. There exists a real constant $K$ such that
\begin{equation*}
  \mean \sum_{T\in\cT(\mathbf{L})}F(T)^x \leq \mean \sum_{T\in\cT(\mathbf{L})}K^{x\nseint{T}}.
\end{equation*}
Note that the number of internal segments in $T$ is equal to the number $\mathbf{k}$ of lines in $\mathbf{L}$. It follows
\begin{equation*}
  \mean \sum_{T\in\cT(\mathbf{L})}F(T)^x \leq \mean \nttl{\mathbf{L}}K^{x\mathbf{k}}.
\end{equation*}
Using the upper bound provided by Equation \eqref{eq:tessellation:number}, one gets
\begin{equation*}
  \mean \sum_{T\in\cT(\mathbf{L})}F(T)^x \leq c_1 + \mean I\left\{\mathbf{k}\geq 2\right\}C^{\mathbf{k}} \left(\frac{\mathbf{k}}{(\log \mathbf{k})^{1-\epsilon}}\right)^{\mathbf{k}-\mathbf{k}/\log \mathbf{k}}K^{x\mathbf{k}}.
\end{equation*}
The constant $c_1$ is related to the sum terms when the number of lines $\mathbf{k}$ is less than $2$. As $F$ is finite, the constant $c_1$ is also finite.
Furthermore, since $\mathbf{L}$ is a Poisson line process with intensity 1, $\mathbf{k}$ is Poisson distributed with parameter $l(D)/\pi$. Let $c_2$ be the normalizing constant of this Poisson distribution. We have
\begin{eqnarray*}
  \mean \sum_{T\in\cT(\mathbf{L})}F(T)^x & \leq & c_1 + c_2
    \sum_{k\geq 2} \left(\frac{l(D)}{\pi}\right)^k \frac{1}{k!} C^{k} 
    \left(\frac{k}{(\log k)^{1-\epsilon}}\right)^{k-k/\log k}K^{xk},\\
  & \leq & c_1+c_2
    \sum_{k\geq 2} \left(\frac{CK^{x}l(D)}{\pi k}\right)^k 
    \left(\frac{k}{(\log k)^{1-\epsilon}}\right)^{k-k/\log k},\\
  & \leq & c_1+c_2
    \sum_{k\geq 2} \left(\frac{CK^{x}l(D)}{\pi(\log k)^{1-\epsilon} }\right)^k 
    \left(\frac{k}{(\log k)^{1-\epsilon}}\right)^{-k/\log k}.\\
\end{eqnarray*}
Now consider $\epsilon=1/2$. The term
\begin{equation*}
  \left(\frac{k}{(\log k)^{1/2}}\right)^{-k/\log k}
\end{equation*}
is upper bounded. Finally it is easy to check that the series
\begin{equation*}
  \sum_{k\geq 2} \left(\frac{CK^{x}l(D)}{\pi(\log k)^{1/2} }\right)^k 
\end{equation*}
is finite.

\section{Proof of Proposition \ref{prop:split:papangelou:mu}}
\label{appendix:proof:split:papangelou}
Let $\TT\sim\mu$ where $\mu$ is the distribution on $\cT$ defined by Equation~\eqref{eq:def:mu}. In view of Definition~\ref{def:split:campbell} and Equation~\eqref{eq:def:mu}, the split Campbell measure of $\TT$ can be written as
\begin{equation*}
  \CSplitMeas(\phi) = \mean \sum_{T\in\cT(\mathbf{L})} 
  \sum_{m\in\M{\TT}}\phi\left(m,T\setminus\{m\}\right).
\end{equation*}
The sum on non-blocking segments can be replaced by a sum over all lines $l$ in $\mathbf{L}$ combined with an indicator function. Below let $s(l,\TT)$ be the segment of $\TT$ lying on the line $l$ of $\mathbf{L}$.
\begin{equation*}
  \CSplitMeas(\phi) = \mean \sum_{T\in\cT(\mathbf{L})} 
  \sum_{l\in\mathbf{L}} \mathbf{1}_{\left\{s(l,T)\text{ is non-blocking}\right\}} \phi\left(s(l,T),T\setminus\{s(l,T)\}\right).
\end{equation*}
For a given line configuration $L$ and a given line $l\in L$, consider the map $T\mapsto \tilde{T}=T\setminus s(l,T)$. If $s(l,T)$ is non-blocking, $T\setminus s(l,T)$ is still a T-tessellation. Also note that $\tilde{T}\in\cT(L\setminus l)$. Furthermore, this map is not injective. Given a $\tilde{T}\in\cT(L\setminus l)$, the subset of tessellations $T$ such that $\tilde{T}=T\setminus s(l,T)$ is obtained by running through all cells $c$ of $\tilde{T}$ hitting the line $l$ and by splitting $\tilde{T}$ by the line segment $l\cap c$. Hence we have
\begin{eqnarray*}
  \lefteqn{
    \sum_{T\in\cT(\mathbf{L})} \mathbf{1}_{\left\{s(l,T)\text{ is non-blocking}\right\}}
    \phi\left(s(l,T),T\setminus\{s(l,T)\}\right)} \\
  & = &
  \sum_{\tilde{T}\in\cT(\mathbf{L}\setminus l)} \sum_{c\in C(\tilde{T})}\mathbf{1}_{\left\{l\cap c\neq\emptyset\right\}}
  \phi\left(l\cap c,\tilde{T}\right).
\end{eqnarray*}
It follows
\begin{equation*}
  \CSplitMeas(\phi) = \mean \sum_{l\in\mathbf{L}} \sum_{T\in\cT(\mathbf{L}\setminus l)} \sum_{c\in C(T)}
  \mathbf{1}_{\left\{l\cap c\neq\emptyset\right\}} \phi\left(l\cap c,T\right).
\end{equation*}
The right-hand side can be rewritten as
\begin{equation*}
  \mean \sum_{l\in\mathbf{L}} \psi\left(l,\mathbf{L}\setminus l\right).
\end{equation*}
Since $\mathbf{L}$ is a Poisson process on the space of lines, its Campbell measure has a simple decomposition with a Papangelou kernel equal to its intensity:
\begin{equation*}
  \mean \sum_{l\in\mathbf{L}} \psi\left(l,\mathbf{L}\setminus l\right) =
  \mean \int \psi\left(l,\mathbf{L}\right)\; dl.
\end{equation*}
Hence the split Campbell measure for $\TT$ can be written as
\begin{equation*}
  \CSplitMeas(\phi) = \mean \int \sum_{T\in\cT(\mathbf{L})} \sum_{c\in C(T)}
  \mathbf{1}_{\left\{l\cap c\neq\emptyset\right\}} \phi\left(l\cap c,T\right)\;dl.
\end{equation*}
In view of Equation~\eqref{eq:def:mu}, the right-hand side can be rewritten as
\begin{equation*}
  \CSplitMeas(\phi) = \mean \int \sum_{c\in C(\TT)}
  \mathbf{1}_{\left\{l\cap c\neq\emptyset\right\}} \phi\left(l\cap c,\TT\right)\;dl.
\end{equation*}
Finally, using the measure $dS$ on $\S{T}$ defined by Equation~\eqref{eq:def:dS}, the right-hand side is expressed as
\begin{equation*}
  \CSplitMeas(\phi) = \mean \int_{\S{\TT}} \phi\left(S,\TT\right)\;dS.
\end{equation*}
Proposition~\ref{prop:split:papangelou:mu} follows from the definition of the split Papangelou kernel.

\section{Proof of Proposition \ref{prop:flip:papangelou:mu}}
\label{appendix:proof:flip:papangelou}
Let $\TT\sim\mu$ where $\mu$ is the distribution on $\cT$ defined by Equation~\eqref{eq:def:mu}. In view of Definition~\ref{def:flip:campbell}, the flip Campbell measure of $\TT$ is defined by
\begin{equation*}
  \CFlipMeas(\phi) = \mean \sum_{T\in\cT(\mathbf{L})} 
  \sum_{F\in\F{T}}\phi(F^{-1},FT).
\end{equation*}
The double sum can be rewritten using the change of variables
\begin{equation*}
  (F,T) \mapsto \left(\tilde{F}=F^{-1},\tilde{T}=FT\right).
\end{equation*}
Note that a  flip does not change the lines supporting a T-tessellation: $L(T)=L(\tilde{T})$, i.e.\ $T\in\cT(L) \Leftrightarrow \tilde{T}\in\cT(L)$. Furthermore, the flip $\tilde{F}$ can be applied to $FT=\tilde{T}$: $\tilde{F}\in\F{\tilde{T}}$. Hence the flip Campbell measure can be expressed as follows
\begin{equation*}
  \CFlipMeas(\phi) = \mean \sum_{\tilde{T}\in\cT(\mathbf{L})} 
  \sum_{\tilde{F}\in\F{\tilde{T}}}\phi(\tilde{F},\tilde{T})
  = \mean \sum_{F\in\F{\TT}} \phi(F,\TT)
\end{equation*}
Therefore the flip Papangelou kernel equals $1$ for any pair $(T,F)$.

\section{Proof of Proposition \ref{prop:symmetric:measures}}
\label{appendix:proof:prop:symmetric:measures}
For a given $A\subset\Te_\text{sm}^2$, let
\begin{equation*}
  \check{A} = \left\{(T,\tilde{T}): (\tilde{T},T)\in A\right\}.
\end{equation*}
It must be shown that $\mu_\text{sm}(\check{A}) = \mu_\text{sm}(A)$. In view of Equation~\eqref{eq:def:sym:meas:sm},
\begin{equation*}
  \mu_\text{sm}(\check{A}) = \mean \int_{\mathcal{S}_{\TT}} 
  \mathbf{1}_A(S\TT,\TT)\;dS  + \mean \sum_{M\in\mathcal{M}_{\TT}} 
  \mathbf{1}_A(M\TT,\TT).
\end{equation*}
Consider the first term of the right-hand side and let $\phi(S,T)=\mathbf{1}_A(ST,T)$. Next use identity~\eqref{eq:alt:split:gnz:mu}. It can be checked that $\phi(M^{-1},MT)=\mathbf{1}_A(T,MT)$. Therefore
\begin{equation*}
  \mean \int_{\mathcal{S}_{\TT}} \mathbf{1}_A(S\TT,\TT)\;dS =
  \mean \sum_{M\in\mathcal{M}_{\TT}} \mathbf{1}_A(\TT,M\TT).
\end{equation*}
Now rewrite the second term of the right-hand side as
\begin{equation*}
  \mean \sum_{M\in\mathcal{M}_{\TT}} \mathbf{1}_A(M\TT,\TT) =
  \mean \sum_{M\in\mathcal{M}_{\TT}} \mathbf{1}_A(M\TT,M^{-1}M\TT). 
\end{equation*}
Use identity~\eqref{eq:alt:split:gnz:mu} with $\phi(M^{-1},MT)=\mathbf{1}_A(MT,M^{-1}MT)$. It can be checked that $\phi(S,T)=\mathbf{1}_A(T,ST)$. Therefore the second term of the right-hand side is equal to
\begin{equation*}
  \mean \int_{\mathcal{S}_{\TT}} \mathbf{1}_A(\TT,S\TT)\; dS.
\end{equation*}
Hence,
\begin{equation*}
  \mu_\text{sm}(\check{A}) = \mean \sum_{M\in\mathcal{M}_{\TT}} \mathbf{1}_A(\TT,M\TT) + 
  \mean \int_{\mathcal{S}_{\TT}} \mathbf{1}_A(\TT,S\TT)\; dS.
\end{equation*}
Compare with \eqref{eq:def:sym:meas:sm}.

\section{Proof of Proposition \ref{prop:irreducibility}}
\label{appendix:proof:prop:irreducibility}
Below it is proved that
under the conditions of Proposition~\ref{prop:irreducibility},
$\mathbf{T}_n$ is $\psi$-irreducible where $\psi$ is the measure on
$\Te$ defined by $\psi(B)=1$ if $B$ contains the empty tessellation $T_D$ and $\psi(B)=0$
otherwise. One needs to prove that for any T-tessellation $T$ there
exists an integer $m\geq 1$ such that
\begin{equation}
  \label{eq:def:irreducibility}
  \prob{\mathbf{T}_m=T|\mathbf{T}_0=T_D}>0.
\end{equation}
First, let us consider the case where $T\ne T_D$. In view of
Proposition~\ref{prop:smf:complet:1}, there exists a sequence
$(T_n)_{n=0,\ldots,m}$ of T-tessellations such that $T_n=T$,
$T_0=T_D$ and $T_{n+1}=U_{n}T_{n}$ where $U_n$ is either a merge or
a flip for any $n=0,\ldots,m-1$. Note that whenever $U_n$ is a
merge, it can be assumed to check condition
\eqref{eq:cond:1:irreducibility}. The probability
\eqref{eq:def:irreducibility} is greater than or equal to the
probability that
\begin{equation*}
  (\mathbf{T}_n)_{n=0,\ldots,m} = (T_n)_{n=0,\ldots,m}.
\end{equation*}
Since $\mathbf{T}_n$ is Markovian, the probability of the event
above is equal to the product of conditional probabilities
\begin{equation*}
  \prob{\mathbf{T}_{n+1}=U_nT_n|\mathbf{T}_n=T_n} =
  p_{\text{t}_n}(T_n)q_{\text{t}_n}(T_n,U_n)\min\{1,r_{\text{t}_n}(T_n,U_n)\},
\end{equation*}
where $\text{t}_n=\text{m}$ if $U_n$ is a merge and
$\text{t}_n=\text{f}$ if $U_n$ is a flip. If $U_n$ is a merge,
condition~\eqref{eq:cond:1:irreducibility} implies that
\begin{eqnarray*}
  p_{\mt}(T_n)q_{\mt}(T_n,U_n) & > & 0,\\
  r_{\mt}(T_n,U_n) & > & 0.
\end{eqnarray*}
If $U_n$ is a flip, condition~\eqref{eq:cond:2:irreducibility} for
$T=T_n$ and $F=U_n$ implies that
\begin{equation*}
  p_{\ft}(T_n)q_{\ft}(T_n,U_n) > 0,
\end{equation*}
while condition~\eqref{eq:cond:2:irreducibility} for $T=U_nT_n$ and
$F=U_n^{-1}$ implies that
\begin{equation*}
  r_{\ft}(T_n,U_n) > 0.
\end{equation*}
Therefore there exists a $m\geq 1$ such that
\eqref{eq:def:irreducibility} holds for any T-tessellation $T\neq
T_D$.

Now let us consider the case where $T=\mathbf{T}_0=T_D$. If either
$p_\mt(T_D)>0$ or $p_\ft(T_D)>0$, a merge or a flip of
$\mathbf{T}_0$ is considered with a probability greater than zero.
However no such update is applicable to $T_D$ and
$\mathbf{T}_1=\mathbf{T}_0=T_D$ almost surely. Therefore the
probability that $\mathbf{T}_1=T_D$ is not zero. If both
$p_\mt(T_D)=p_\ft(T_D)=0$, a split of $\mathbf{T}_0$ is considered.
If rejected, $\mathbf{T}_1=T_D$, otherwise $\mathbf{T}_1$ is a
tessellation with $2$ cells and as shown above there exists a finite
integer $m$ such that the probability that the Markov chain reaches
the empty tessellation $T_D$ starting from $\mathbf{T}_1$ is
positive.

\end{document}